\documentclass[draft,reqno]{amsart}
\usepackage{amsmath,amsfonts,amssymb,amscd,verbatim,multicol}
\usepackage[arrow,matrix,cmtip,curve]{xy}
\usepackage{lscape}
\usepackage{mathrsfs}  
\usepackage{fullpage}
 
\numberwithin{equation}{section}

\theoremstyle{definition}
\newtheorem{theorem}{Theorem}[section]

\newtheorem{corollary}[theorem]{Corollary}

\newtheorem{lemma}[theorem]{Lemma}
\newtheorem{proposition}[theorem]{Proposition}

\newtheorem{remark}[theorem]{Remark}

\newtheorem{thm}{Theorem} 

\newtheorem{hypothesis}{Hypothesis}

\DeclareMathOperator\gk{GK.dim}

\DeclareMathOperator\Span{Span}

\newcommand\bs{\backslash}

\newcommand\inv{^{-1}}
\newcommand\iso{\cong}
\newcommand\kk{\Bbbk}
\newcommand\wh{\widehat}

\newcommand\cA{\mathcal A}

\newcommand\cM{\mathcal M}
\newcommand\cO{\mathcal O}
\newcommand\cS{\mathcal S}

\newcommand{\qp}{\cO_q(\kk^2)}
\newcommand{\qpp}{\cO_p(\kk^2)}

\newcommand\bp{\mathbf p}
\newcommand\bq{\mathbf q}

\newcommand\bx{\mathbf x}
\newcommand\by{\mathbf y}

\newcommand{\qnp}{\cO_{\bp}(\kk^n)}
\newcommand{\qnq}{\cO_{\bq}(\kk^n)}

\newcommand{\qmq}{\cO_{\bq}(\kk^m)}
\newcommand{\qma}{\cO_{\lambda,\bq}(\cM_n(\kk))}
\newcommand{\qmamu}{\cO_{\mu,\bq}(\cM_n(\kk))}
\newcommand{\qmamum}{\cO_{\mu,\bq}(\cM_m(\kk))}
\newcommand{\pma}{\cO_{\lambda,\bp}(\cM_n(\kk))}

\newcommand{\wa}{A_1^q(\kk)}
\newcommand{\wap}{A_1^p(\kk)}
\newcommand{\hwa}{H(\wa)}
\newcommand{\hwap}{H(\wap)}

\newcommand{\wam}{A_m^{\bq,\mu}(\kk)}
\newcommand{\wanp}{A_n^{\bp,\gamma}(\kk)}
\newcommand{\hw}{H_n^{\bq,\mu}}
\newcommand{\hwm}{H_m^{\bq,\mu}}
\newcommand{\hwp}{H_n^{\bp,\gamma}}

\newcommand\ep{\varepsilon}

\begin{document}

\title{The isomorphism problem for quantum affine spaces,
homogenized quantized Weyl algebras, and quantum matrix algebras}
%\title{Isomorphism problems in noncommutative algebra}

\author{Jason Gaddis}
\address{Department of Mathematics and Statistics, 
P.O. Box 7388, Wake Forest University, Winston-Salem, NC 27109} 
\email{gaddisjd@wfu.edu}

%\subjclass[2010]{Primary 16W50; Secondary 16T99}

%\date{}

\begin{abstract}
Bell and Zhang have shown that if 
$A$ and $B$ are two connected graded algebras finitely generated 
in degree one that are isomorphic as ungraded algebras,
then they are isomorphic as graded algebras. 
We exploit this result to solve the isomorphism problem
in the cases of quantum affine spaces, quantum matrix algebras,
and homogenized multiparameter quantized Weyl algebras.
Our result involves determining the degree one normal elements,
factoring out, and then repeating.
This creates an iterative process that allows one to
determine relationships between relative parameters.
\end{abstract}

\maketitle

\section{Introduction}

Throughout, $\kk$ is a field and all algebras are $\kk$-algebras. 
All isomorphisms should be read as `isomorphisms as $\kk$-algebras'.
Our primary source for all definitions is \cite{BG}.

\begin{hypothesis}
\label{hyp.cga}
$A$ is a connected graded algebra finitely
generated over $\kk$ in degree $1$.
\end{hypothesis}

Let $R$ and $S$ be algebras satisfying Hypothesis \ref{hyp.cga}
with bases $\{x_i\}$ and $\{y_i\}$, respectively.
Then $R$ and $S$ are isomorphic as graded algebras
if there exists an algebra isomorphism $\Phi:R \rightarrow S$
such that $\Phi(x_i) = \sum \alpha_{ij} y_j$, $\alpha_{ij} \in \kk$,
for each $x_i$.
If $a_{ij}\neq 0$, then we say $y_j$ is a {\sf summand} of $\Phi(x_i)$.
Because of the following result,
we will often assume without comment that isomorphisms
between graded algebras are graded isomorphisms.

\begin{thm}[Bell, Zhang {\cite[Theorem 0.1]{BZ}}]
Let $A$ and $B$ be two algebras satisfying Hypothesis \ref{hyp.cga}.
If $A \iso B$ as ungraded algebras, then $A \iso B$ as graded algebras.
\end{thm}
 
A square matrix $\bp = (p_{ij}) \in\cM_n(\kk^\times)$ is 
{\sf multiplicatively antisymmetric} if 
$p_{ii} = 1$ and $p_{ij}=p_{ji}\inv$ for all $i \neq j$. 
Let $\cA_n \subset \cM_n(\kk^\times)$ be the subset of 
multiplicatively antisymmetric matrices. 
A matrix $\bq \in \cA_n$ is a {\sf permutation} of $\bp$
if there exists a permutation $\sigma \in \cS_n$
such that $q_{ij}=p_{\sigma(i)\sigma(j)}$ for all $i,j$.

For $\bp \in \cA_n$, the {\sf (multiparameter) quantum affine $n$-space} 
$\qnp$ is defined as the algebra with basis $\{x_i\}$, 
$1 \leq i \leq n$, subject to the relations $x_ix_j=p_{ij}x_jx_i$ 
for all $1 \leq i,j \leq n$. 

\begin{thm}[Theorem \ref{thm.qas}]
$\qnp \iso \qmq$ if and only if $m=n$ and $\bp$ is a permutation of $\bq$.
\end{thm}

Mori proved Theorem \ref{thm.qas} when $n=3$ \cite[Example 4.10]{M} 
and this was extended to all $n$ by Vitoria \cite[Lemma 2.3]{V}.
We present a simple, self-contained proof that does not 
rely on the noncommutative projective algebraic geometry associated to $\qnp$.

{\sf (Multiparameter) quantized Weyl algebras} may be regarded 
as $\gamma$-difference operators on $\qnp$. 
Let $\bp \in \cA_n$ and 
$\gamma=(\gamma_1,\hdots,\gamma_n) \in (\kk^\times)^n$. 
Then $\wanp$ is the algebra with basis $\{x_i,y_i\}$,
$1 \leq i \leq n$, subject to the relations
\begin{align*}
	y_iy_j &= p_{ij} y_jy_i & & (\text{all } i,j) \\
	x_ix_j &= \gamma_i p_{ij} x_jx_i & & (i < j) \\
	x_iy_j &= p_{ji} y_jx_i & & (i < j) \\
	x_iy_j &= \gamma_j p_{ji} y_jx_i & & (i > j) \\
	x_jy_j &= 1 + \gamma_j y_jx_j + \sum_{l<j}(\gamma_l-1)y_lx_l & & 
			(\text{all } j).
\end{align*}
In the case of $n=1$, the parameter $\bp$ plays no role 
and so we refer to the single parameter simply as $p$ and write $\wap$.
By \cite[Section 6]{G}, $\wap \iso \wa$ if and only if $p=q^{\pm 1}$.
Goodearl and Hartwig solved the isomorphism problem for 
quantized Weyl algebras when no $\gamma_i$ 
is a root of unity \cite[Theorem 5.1]{GH}.
They proved that if $\wanp \iso \wam$,
then $n=m$ and there exists a sign vector $\ep \in \{\pm 1\}^n$ such that 
\begin{align}
\label{eq.hwa1}
\mu_i = \gamma_i^{\ep_i}
\end{align} 
and $\bp,\bq$ satisfy
\begin{align}
\label{eq.hwa2}
q_{ij} = \begin{cases}
p_{ij}	& \text{if } (\ep_i,\ep_j) = (1,1), \\
p_{ji}	& \text{if }(\ep_i,\ep_j) = (-1,1), \\
\gamma_i\inv p_{ji}	& \text{if } (\ep_i,\ep_j) = (1,-1), \\
\gamma_i p_{ij}	& \text{if } (\ep_i,\ep_j) = (-1,-1).
\end{cases}
\end{align}
The isomorphisms they give hold regardless of 
the root-of-unity condition.
Levitt and Yakimov have extended this result to 
the case where all parameters are roots of unity
and $\wanp, \wam$ are free over their centers
by utilizing noncommutative discriminants \cite[Corollary 6.4]{LY}.
Several intermediate cases are still open.

The quantized Weyl algebras are not graded and so we consider their homogenizations. 
The {\sf homogenized (multiparamenter) quantized Weyl algebra} 
$\hwp$ has algebra basis $\{z,x_i,y_i\}$, $1 \leq i \leq n$, 
where $z$ commutes with the $x_i$ and $y_i$.
The relations in $\hwp$ are the same as those for $\wanp$
except the final relation type is replaced by its homogenization,
\[ x_jy_j = z^2 + \gamma_j y_jx_j + \sum_{l<j}(\gamma_l-1)y_lx_l.\]
The isomorphisms defined by Goodearl and Hartwig extend to isomorphisms in
the homogenized case by fixing the homogenizing variable.
The next theorem suggests that the Goodearl-Hartwig
result should hold in general.

\begin{thm}[Theorem \ref{thm.hwaiso}]
Let $\gamma,\mu \in (\kk^\times)^n$ and $\bp,\bq \in \cA_n$
Then $\Phi:\hwp \rightarrow \hwm$ is an isomorphism 
if and only if $n=m$ and there exists $\ep \in \{\pm 1\}^n$ such that
$\gamma$ and $\mu$ satisfy \eqref{eq.hwa1}
while $\bp$ and $\bq$ satisfy \eqref{eq.hwa2}.
\end{thm}

Fix parameters $\lambda \in \kk^\times$ and $\bp \in \cA_n$.
The {\sf (multiparameter) quantum ($n\times n$) matrix algebra}, $\pma$, 
is the algebra with basis $\{X_{ij}\}$, $1 \leq i,j \leq n$,
subject to the relations
\begin{align*}
X_{lm}X_{ij} = \begin{cases}
	p_{li}p_{jm}X_{ij}X_{lm} + (\lambda-1)p_{li}X_{im}X_{lj} & l > i, m > j\\
		\lambda p_{li}p_{jm} X_{ij}X_{lm} & l > i, m \leq j \\
		p_{jm}X_{ij}X_{lm} & l = i, m > j.			
	\end{cases}
\end{align*}

By \cite[Theorem 2]{AST}, $\gk(\pma) = n^2$ if and only if $\lambda \neq -1$.
We show in Theorem \ref{thm.pma} that the ideal of $\qma$ 
generated by degree one normal elements is $\langle X_{1n},X_{n1}\rangle$
when $\lambda \neq 1$.
On the other hand, every degree one generator of $\qmq$ is normal.
This, combined with the fact that $\gk(\qmq)=m$
(\cite[Lemma II.9.7]{BG}) proves that
$\pma \iso \qmq$ if and only if $m=n^2$ and $\lambda = 1$.
Hence, we ignore the cases $\lambda = \pm 1$ henceforth.

The isomorphism problem in the single parameter case of $\pma$ for all $n$
was solved in \cite[Proposition 3.1]{G},
as was the multiparameter case for $n=2$ \cite[Proposition 4.2]{G}.
We extend these results in the following theorem.

\begin{thm}[Theorem \ref{thm.mpqma}]
$\pma \iso \qmamum$ if and only if
$n=m$ and one of the following cases holds:
\begin{enumerate}
\item $\lambda=\mu$ and $\bp = \bq$;
\item $\lambda=\mu$ and $p_{ij}=\lambda\inv q_{ji}$
for all $i,j$;
\item $\lambda=\mu\inv$ and $p_{ij}=q_{n+1-i,n+1-j}$ for all $i,j$;
\item $\lambda=\mu\inv$ and $p_{ij}=\lambda\inv q_{n+1-j,n+1-i}$ for all $i,j$.
\end{enumerate}
\end{thm}

\section{Quantum affine spaces}
\label{sec.qas}

\begin{lemma}
\label{lem.taudef}
Let $\Phi:R \rightarrow S$ be a graded isomorphism 
between algebras satisfying Hypothesis \ref{hyp.cga}
with homogeneous generators 
$\{x_i\}$ and $\{y_i\}$, respectively, $1 \leq i \leq n$. 
Then there exists a permutation $\tau \in \cS_n$ such that
$y_{\tau(i)}$ is a summand of $\Psi(x_i)$ for each $i$.
\end{lemma}

\begin{proof}
Write $\Phi(x_i)=\sum \gamma_{ij} y_j$ and let $M = (\gamma_{ij})_{ij}$.
Because $\Phi$ is an isomorphism, $\det(M) \neq 0$.
For $k$, $1 \leq k \leq n$, let $I=\{1,\hdots,k\}$. 
Choose $J \subset \{1,\hdots,n\}$ such that
$|J|=k$ and such that the minor $M_{IJ} \neq 0$.
We claim there exists an injective map of sets $\tau:I \rightarrow J$.

If $k=1$ and $M_{1j}$ denotes the $(1,j)$-minor of $M$, then we have
\[ 0 \neq \det(M) = \sum_{j=1}^{n} (-1)^{j+1} \gamma_{1j} M_{1j}.\]
Hence, there exists $j$ such that $\gamma_{1j} M_{1j} \neq 0$. Set $\tau(1)=j$.

Suppose inductively that the claim holds for some $k<n$.
Since $\det(M) \neq 0$, there exists a nonzero minor 
$N$ of size $(k+1) \times (k+1)$.
Let $I',J' \subset \{1,\hdots,n\}$ such that $M_{I'J'}=N$, so $|I'|=|J'|=k+1$.
As above, we can choose $\ell$ such that $\gamma_{i_1 j_\ell} N_{1 j_\ell} \neq 0$
and set $\tau(i_1)=j_\ell$.
Now apply the inductive hypothesis to 
$I=\{i_2,\hdots,i_{k+1}\}$ and $J=\{j_1,\hdots,\wh{j_\ell},\hdots,j_{k+1}\}$.
The result follows by setting $N=M$.
\end{proof}

For the remainder of this section, let $\{x_i\}$, $\{y_i\}$ 
be bases for $\qnp$ and $\qnq$, respectively.
and suppose $\Phi:\qnp \rightarrow \qnq$ is a isomorphism. 
We claim $\bp = \tau.\bq$ where $\tau \in \cS_n$
is determined by Lemma \ref{lem.taudef}.

\begin{lemma}
\label{lem.qparam1}
If $r,s \in \{1,\hdots,n\}$ such that $p_{rs} \neq 1$, 
then $p_{rs}=q_{\tau(r)\tau(s)}$.
\end{lemma}

\begin{proof}
Write $\Phi(x_r)=\sum \alpha_i y_i$ and $\Phi(x_s)=\sum \beta_i y_i$. Then, 
\[ 0 	= \Phi(x_rx_s - p_{rs} x_sx_r) 
	= (1-p_{rs})\left(\sum_{d=1}^n \alpha_d \beta_d y_d^2\right) 
		+ \sum_{1 \leq i \neq j \leq n} 
		\left(\alpha_i \beta_j -p_{rs} \alpha_j\beta_i \right) y_iy_j.
\]
Since $p_{rs} \neq 1$, then $\alpha_d=0$ or $\beta_d = 0$ for each $d$. Thus,
\begin{align}
0 	&= \Phi(x_rx_s - p_{rs} x_sx_r) \\
	&= \sum_{1 \leq i < j \leq n} 
	\left[ (\alpha_i \beta_j -p_{rs} \alpha_j\beta_i) 
	+ q_{ji}(\alpha_j\beta_i - p_{rs}\alpha_i\beta_j)\right] y_iy_j \notag \\
  	&= \sum_{1 \leq i < j \leq n} \left[ (\alpha_j\beta_i(q_{ji}-p_{rs}) 
  	+ \alpha_i\beta_j (1 - q_{ji}p_{rs})\right]y_iy_j. \label{not1}
\end{align}
By Lemma \ref{lem.taudef}, $\alpha_{\tau(r)}$, $\beta_{\tau(s)} \neq 0$. 
Thus, $\alpha_{\tau(s)}=0$ and $\beta_{\tau(r)}=0$. 
If $\tau(r) > \tau(s)$, then by \eqref{not1} 
the coefficient of $y_{\tau(s)}y_{\tau(r)}$ is 
$\alpha_{\tau(r)}\beta_{\tau(s)}(q_{\tau(r)\tau(s)} - p_{rs})$. 
Therefore, $p_{rs} = q_{\tau(r)\tau(s)}$. 
One the other hand, if $\tau(r) < \tau(s)$, 
then the coefficient of $y_{\tau(r)}y_{\tau(s)}$ is $\alpha_{\tau(r)}\beta_{\tau(s)}(1-q_{\tau(s)\tau(r)}p_{rs})$. 
Therefore, $p_{rs} = q_{\tau(s)\tau(r)}\inv = q_{\tau(r)\tau(s)}$. 
Because $p_{rs} \neq 1$, then $r \neq s$ and so, 
because $\tau$ is a permutation, $\tau(r) \neq \tau(s)$.
\end{proof}
\begin{lemma}
\label{lem.qparam2}
If $r,s \in \{1,\hdots,n\}$ such that $p_{rs} = 1$, 
then $p_{rs}=q_{\tau(r)\tau(s)}$.
\end{lemma}

\begin{proof}
Define $\bp^{\#}=\{p_{ij} \in \bp \mid p_{ij} \neq 1\}$
and $\bq^{\#}$ similarly. 
By Lemma \ref{lem.qparam1}, $\bp^{\#} \leq \bq^{\#}$. 
Because $\Phi$ is an isomorphism, 
then we can apply Lemma \ref{lem.qparam1} to $\Phi\inv$ 
to get that $\bq^{\#} \leq \bp^{\#}$. 
Thus, $\bp^{\#} = \bq^{\#}$.
It follows that $p_{rs}=1$ implies $q_{\tau(r)\tau(s)}=1$
so $p_{rs} = q_{\tau(r)\tau(s)}$ for all $r,s \in \{1,\hdots,n\}$.
\end{proof}

\begin{theorem}
\label{thm.qas}
$\qnp \iso \qmq$ if and only if $m=n$ and $\bp$ is a permutation of $\bq$.
\end{theorem}

\begin{proof}
Suppose $n=m$ and there exists $\sigma \in \cS_n$ such that $\bp=\sigma.\bq$.
Define a map $\Psi: \qnp \rightarrow \qnq$ by $x_i \mapsto y_{\sigma(i)}$.
For all $i,j$, $1 \leq i,j \leq n$,
\[ \Psi(x_i)\Psi(x_j) - p_{ij}\Psi(x_j)\Psi(x_i) 
	= y_{\sigma(i)}y_{\sigma(j)} - q_{\sigma(i)\sigma(j)}y_{\sigma(j)}y_{\sigma(i)}=0. \]
Hence, $\Psi$ extends to a bijective homomorphism. 
Thus, $\qnp \iso \qnq$.

If $\qnp \iso \qmq$, then $n=\gk(\qnp)=\gk(\qmq)=m$, so $n=m$.
By Lemmas \ref{lem.taudef}, \ref{lem.qparam1}, and \ref{lem.qparam2}
there exists a permutation $\tau \in \cS_n$ such that $\bp = \tau.\bq$.
\end{proof}

\section{Degree one normal elements}
\label{sec.deg1}

In order to consider the isomorphism problem for
homogenized quantized Weyl algebras and quantum matrix algebras,
we identify homogeneous degree one normal elements.

Let $A$ be an algebra satisfying Hypothesis \ref{hyp.cga}.
We define the ideals the ideals
$I_0 \subset I_1 \subset \cdots \subset I_n$ of $A$
where $I_k/I_{k-1}$ is generated by all 
(homogeneous) degree one normal elements in $A/I_{k-1}$.
By convention we set $I_k=0$ for $k<0$.
We frequently identify elements in $A$ with their images in $A/I_{k-1}$.
An algebra $B = A/I_k$ for some $k$ is an
{\sf iterative quotient of $A$ (by degree one normal elements)}.

It is clear that $B$ also satisfies Hypothesis \ref{hyp.cga}.
Moreover, if $\Phi:A \rightarrow A'$ is an
isomorphism of algebras satisfying Hypothesis \ref{hyp.cga},
then $A/I_k \iso A'/\Phi(I_k)$ for all $k > 0$.

Our goal is to identify degree one normal
elements in each such quotient.

\begin{theorem}
\label{thm.wanorm}
In the case of $\hwp$,
$I_0 = \langle z \rangle$ and for $0 < k \leq n$,
$I_k/I_{k-1} = \langle x_k,y_k \rangle$.
\end{theorem}

\begin{proof}
Suppose $u \in \hwp$
is a degree one normal element and write 
\[u = a z + \sum_{i=1}^m (\alpha_{1i} x_i + \alpha_{2i} y_i),\]
with $a,\alpha_i,\beta_i \in \kk$ and $m \leq n$.
We may assume $\alpha_{1m}$ or $\alpha_{2m}$ is nonzero.
Both cases are similar so assume the former. Then
\[
u y_m = ay_m z 
	+ \alpha_{1m} \left(z^2 + \gamma_m y_m x_m 
	+ \sum_{l < m} (\gamma_l - 1) y_l x_l \right)
	+ \sum_{i=1}^{m-1} \alpha_{1i} (p_{mi} y_m x_i) 
	+ \sum_{i=1}^m \alpha_{2i} y_iy_m.
\]
By normality, there exists $r \in \hwp$ such that $uy_m = ru$. Write
\[  r = b z + \sum_{i=1}^m (\beta_{1i} x_i + \beta_{2i} y_i).\]
Then
\begin{align*}
ru  
%	&= \left( b z + \sum_{i=1}^m (\beta_{1i} x_i + \beta_{2i} y_i\right)
%	\left(a z + \sum_{i=1}^m (\alpha_{1i} x_i + \alpha_{2i} y_i) \right) \\
	&=  \sum_{i=1}^m \left( 
		(b\alpha_{1i} + a\beta_{1i} )x_iz + 
		(b\alpha_{2i} + a\beta_{2i} )y_iz 
		\right) +
		\sum_{i=1}^m (\alpha_{1i}\beta_{1i} x_i^2 + \alpha_{2i}\beta_{2i} y_i^2) \\
		&\quad + \sum_{i=1}^{m-1} \left( 
		(\alpha_{1i}\beta_{1m} + \gamma_i p_{im} \alpha_{1m}\beta_{1i}) x_ix_m +
		(\alpha_{2i}\beta_{2m} + p_{im} \alpha_{2m}\beta_{2i}) y_iy_m \right) \\
		&\quad + \sum_{i=1}^{m-1} \left( 
		(\alpha_{1i}\beta_{2m} + p_{mi}\alpha_{2m}{\beta_1i}) y_mx_i +
		(\alpha_{1m}\beta_{2i} + \gamma_i p_{im} \alpha_{2i}\beta_{1m}) y_ix_m
		\right) \\
		&\quad + \text{(additional terms in $y_ix_i$ and $z^2$)}.
\end{align*}

We now evaluate several coefficients in $0=ru-uy_m$.
The coefficient of $x_m^2$ is $\alpha_{1m}\beta_{1m}$, so $\beta_{1m}=0$
by our hypothesis on $\alpha_{1m}$. 
It follows that the coefficient of $x_ix_m$ is
$\gamma_i p_{im}\alpha_{1m}\beta_{1i}$, so $\beta_{1i}=0$ for all $i=1,\hdots,m$. 
The coefficient of $x_m z$ is $b\alpha_{1m}$ so $b=0$. 
Finally, it follows that the coefficient of $y_ix_m$ for $i<m$
is $\alpha_{1m}\beta_{2i}$ so $\beta_{2i}=0$ for $i<m$. Thus
\[
ru 	= \beta_{2m} y_m \left( a z + \sum_{i=1}^m (\alpha_{1i} x_i + \alpha_{2i} y_i) \right)
	= \beta_{2m} \left( ay_mz + \sum_{i=1}^m (\alpha_{1i} y_mx_i + \alpha_{2i} p_{mi} y_iy_m ) \right).
\]
But then the coefficient of $z^2$ in $ru-uy_m$ is $-\alpha_{1m}$, a contradiction.
Since $z \in I_0$ it follows that $I_0=(z)$.

The general statement is similar.
Let $B=\hwp/I_{k-1}$ be an iterative quotient. Write
\[ v = \sum_{i=k}^m (\alpha_{1i} x_i + \alpha_{2i} y_i) + I_{k-1}.\]
It is clear that $x_k,y_k \in I_k$ and an analysis as
above shows that $v = \alpha_{1k} x_k + \alpha_{2k}y_k$.
\end{proof}

\begin{corollary}
\label{cor.hwprops}
$\hwp$ is Artin-Schelter regular of global and GK dimension $2n+1$. 
\end{corollary}

\begin{proof}
Observe that $\hwp/I_{n-1} \iso \qp$, the quantum plane with $q=\gamma_1$,
and $\qp$ is Artin-Schelter regular of global and GK dimension $2$ \cite{AS}.
By Theorem \ref{thm.wanorm}, $(z,x_1,y_1,\hdots,x_{n-1},y_{n-1})$
is a regular normalizing sequence.
Hence, $\hwp$ has the required properties by \cite[Lemma 5.7 and Corollary 5.10]{L}.
\end{proof}

\begin{remark}
\label{rmk.pma}
Observe from the defining relations of $\pma$
that for every pair of generators $X_{im},X_{lj}$
with $l > i$ and $m > j$ there exists a unique pair
$X_{ij},X_{lm} \notin \{X_{im},X_{lj}\}$ 
such that $X_{im}X_{lj}$ is a linear
combination of $X_{ij}X_{lm}$ and $X_{lm}X_{ij}$.
Moreover, if $l=i$ or $m=j$, then there is no such pair.
\end{remark}

\begin{theorem}
\label{thm.pma}
In the case of $\pma$,
\[I_k/I_{k-1} = \langle 
X_{1(n-k)},X_{(n-k)1},X_{2(n-k+1)},X_{(n-k+1)2},\hdots,X_{(k+1)n},X_{n(k+1)} 
\rangle,
\quad
1 \leq k < n.\]
\end{theorem}

\begin{proof}
Let $B=\pma/I_{k-1}$ and recall that in case $k=0$ we have $I_{k-1}=0$.
It is clear from the defining relations that the given generators
are degree one normal elements in $B$.

Let $u$ be a degree one normal element of $B$.
Suppose by way of contradiction that $u$
has a summand $X_{ij}$ such that 
\[ X_{ij} \notin \{ X_{1(n-k)},X_{(n-k)1},X_{2(n-k+1)},X_{(n-k+1)2},\hdots,X_{(k+1)n},X_{n(k+1)} \}.\]
Hence, there exists $X_{lm}$ such that 
$l>i$ and $m>j$ or $l<i$ and $m<j$.
We consider the first case though the second is similar.
For simplicity, write 
$x_1=X_{ij}$, $x_2=X_{im}$,
$x_3=X_{lj}$, $x_4=X_{lm}$, and
\[u = ax_1 + bx_2 + cx_3 + dx_4 + \bx\]
where $a,b,c,d \in \kk$ and $\bx$ is a degree one element
such that $x_1,\hdots,x_4$ are not summands. 
By hypothesis, $a \neq 0$. Then
\[ x_4 u 
	= a(p_{li}p_{jm} x_1x_4 + (\lambda-1)p_{li}x_2x_3) 
		+ b \lambda p_{li} x_2x_4 + cp_{li}x_3x_4 + dx_4^2 + x_4 \bx.\]
Clearly $x_1x_4,x_2x_4,x_3x_4,x_4^2$ are not summands of $x_4\bx$.
Since $x_1,x_4$ is the unique pair such that $x_2x_3$ is a linear
combination of $x_1x_4$ and $x_4x_1$, and $x_1$ is not a summand of $\bx$,
then by Remark \ref{rmk.pma}, $x_2x_3$ is also not a summand of $x_4\bx$.

Because $u$ is normal, there exists $r \in B$ such that $x_4u=ur$. Write 
\[ r = a'x_1 + b'x_2 + c'x_3 + d'x_4 + \by\]
where $a',b',c',d' \in \kk$ and $\by$ is a degree one element
such that $x_1,\hdots,x_4$ are not summands.

On the other hand,
\begin{align*}
ur &= (ax_1 + bx_2 + cx_3 + dx_4 + \bx)(a'x_1 + b'x_2 + c'x_3 + d'x_4 + \by) \\
	&= (aa' x_1^2 + bb' x_2^2 + cc' x_3^2 + dd' x_4^2)
		+ (ab' + p_{jm}a'b)x_1x_2 
		+ (ac'+\lambda p_{li} a'c) x_1x_3
		+ (ad' + p_{li}p_{jm}a'd) x_1x_4 \\
		&\qquad+ (bd' + \lambda p_{li}b'd) x_2x_4  
		+ (cd' + p_{jm}c'd)x_3x_4 
		+ (bc' + \lambda p_{li}p_{mj} b'c + (a'd)(\lambda-1))x_2x_3
		+ \bx r + u\by.
\end{align*}
It is clear that $x_1^2$ is not a summand of $x_4\bx$, $\bx r$, or $u\by$.
Hence, because $x_4u=ur$, we must have $0 = aa'$.
But $a\neq 0$ by hypothesis so $a'=0$.
The expression for $ur$ now reduces to
\begin{align*}
ur &= (bb' x_2^2 + cc' x_3^2 + dd' x_4^2)
		+ ab'x_1x_2 
		+ ac'x_1x_3
		+ ad'x_1x_4 \\
		&\qquad+ (bd' + \lambda p_{li}b'd) x_2x_4  
		+ (cd' + p_{jm}c'd)x_3x_4 
		+ (bc' + \lambda p_{li}p_{mj} b'c)x_2x_3
		+ \bx r + u\by.
\end{align*}

Similarly, because $x_2=X_{im}$ and $x_1=X_{ij}$,
then by Remark \ref{rmk.pma}, 
$x_1x_2$ is not a summand of $x_4\bx$, $\bx r$, or $u\by$.
Thus, $ab'=0$ so $b'=0$.
The same logic applies to $x_3$ and $x_1$ so $c'=0$. Hence,
\begin{align*}
ur &= dd' x_4^2 	+ ad'x_1x_4 + bd' x_2x_4  + cd' x_3x_4 
		+ \bx r + u\by.
\end{align*}

Finally, we can apply similar reasoning to 
conclude that $x_2x_3$ is not a summand of $ur$.
Hence, the coefficient of $x_2x_3$ in $x_4u$ must be zero.
It follows that $a(\lambda-1)p_{li}=0$, so $\lambda = 1$,
a contradiction.
\end{proof}

\section{Homogenized quantized Weyl algebras}

The following was proved in \cite[Proposition 5.4.5]{gad-thesis}.
However, in light of Theorem \ref{thm.wanorm},
we give a much simpler proof that also
outlines the strategy for the general case.

\begin{proposition}
$\hwap \iso \hwa$ if and only if $p=q^{\pm 1}$.
\end{proposition}

\begin{proof}
Let $\{X,Y,Z\}$ (resp. $\{x,y,z\}$) be the standard basis of $\hwap$ (resp. $\hwa$).
If $p=q$, then there is nothing to prove.
If $p=q\inv$, then the map given by
$X \mapsto -q y$, $Y \mapsto x$, and $Z \mapsto z$
clearly extends to a bijective homomorphism.

Conversely, suppose $\Phi:\hwap \rightarrow \hwa$ is a graded isomorphism.
Since $\Phi((Z)) = (z)$ by Theorem \ref{thm.wanorm},
then there is an induced isomorphism
\[\qpp \iso \hwap/(Z) \iso \hwa/(z) \iso \qp.\]
By Theorem \ref{thm.qas}, $p=q^{\pm 1}$.
\end{proof}

We now move to the general case.

\begin{lemma}
\label{lem.hwa1}
Suppose $\Phi:\hwp \rightarrow \hw$ is an isomorphism.
For all $k$, $1 \leq k \leq n$,
$\mu_k = \gamma_k^{\pm 1}$.
That is, $\gamma$ and $\mu$ satisfy \eqref{eq.hwa1}.
\end{lemma}

\begin{proof}
Fix $k$, $1 \leq k \leq n$.
Suppose first that $\gamma_k \neq 1$.
By Theorem \ref{thm.wanorm}, 
$I_k = \langle X_k,Y_k \rangle$ is mapped bijectively to 
$J_k = \Phi(I_k) = \langle x_k,y_k \rangle$.
Thus, there exists $a,b,c,d \in \kk$ with $ab-cd \neq 0$ such that
\[
	\Phi(X_k) = ax_k + by_k + J_{k-1} \text{ and }
	\Phi(Y_k) = cx_k + dy_k + J_{k-1}.
\]
Then
\begin{align*}
0 
	&= \Phi(X_kY_k - \gamma_k Y_kX_k + I_{k-1}) \\
	&= (ax_k + by_k)(cx_k + dy_k) 
			- \gamma_k (cx_k + dy_k)(ax_k + by_k) + J_{k-1} \\
	&= (1-\gamma_k)(ac x_k^2 + bd y_k^2)
		+ (ad - \gamma_k bc)(x_ky_k) + (bc-\gamma_k ad) y_kx_k + J_{k-1} \\
	&= (1-\gamma_k)(ac x_k^2 + bd y_k^2)
		+ (ad - \gamma_k bc)(\mu_k y_kx_k) + (bc-\gamma_k ad) y_kx_k + J_{k-1} \\
	&= (1-\gamma_k)(ac x_k^2 + bd y_k^2)
		+ \left[ ad(\mu_k - \gamma_k) + bc(1-\mu_k\gamma_k) \right] y_kx_k + J_{k-1}.
\end{align*}
Because $\gamma_k \neq 1$, then the coefficients of $x_k^2$ and $y_k^2$
must be zero on the above. Hence, $ac=0$ and $bd=0$.
But $ad-bc \neq 0$ and so we have two cases.
In the first case, $b=c=0$ and $a,d \neq 0$ so $\mu_k = \gamma_k$.
In the second case, $a=d=0$ and $b,c \neq 0$ so $\mu_k = \gamma_k\inv$.

If $\mu_k \neq 1$, then an argument as above
shows that $\gamma_k = \mu_k^{\pm 1}$.
Thus, $\gamma_k=1$ if and only if $\mu_k = 1$.
\end{proof}

\begin{lemma}
\label{lem.hwa2}
Suppose $\Phi:\hwp \rightarrow \hw$ is an isomorphism.
Then $\bp$ and $\bq$ satisfy \eqref{eq.hwa2}.
\end{lemma}

\begin{proof}
Let $V_k = \Span_\kk \{z,x_\ell,y_\ell \mid \ell < k\}$.
Because $\Phi$ is assumed to be graded, 
then it follows from Theorem \ref{thm.wanorm} that for $i < j$, 
\[
\Phi(Y_i) = ax_i + by_i + K
\quad\text{and}\quad
\Phi(Y_j) = cx_j + dy_j + L\]
for some $K \in V_i$ and $L \in V_j$.

We have,
\begin{align*}
0 	&= \Phi(Y_iY_j-p_{ij}Y_jY_i) \\
	&= (ax_i + by_i + K)(cx_j + dy_j + L) - p_{ij}(cx_j + dy_j + L)(ax_i + by_i + K) \\
	&= ac(x_ix_j-p_{ij}x_jx_i) + ad(x_iy_j-p_{ij}y_jx_i)
		+ bc(y_ix_j-p_{ij}x_jy_i) + bd(y_iy_j-p_{ij}y_jy_i) \\ 
	&\hspace{5em}	+ (ax_i + by_i + K)L-p_{ij}L(ax_i + by_i + K) \\
	&= ac(\mu_iq_{ij}-p_{ij})x_jx_i + ad(q_{ji}-p_{ij})y_jx_i
		+ bc(1-p_{ij}\mu_iq_{ij})x_jy_i + bd(q_{ij}-p_{ij})y_jy_i \\ 
	&\hspace{5em}	+ (ax_i + by_i + K)L-p_{ij}L(ax_i + by_i + K).
\end{align*}
The defining relations clearly imply that the monomials $x_jx_i$,
$y_jy_i$, $x_jy_i$, and $y_jy_i$ do not appear as summands of
$(ax_i + by_i + K)L-p_{ij}L(ax_i + by_i + K)$.
Hence, the coefficients of those monomials in the above must
each be zero.

Suppose $\gamma_i,\gamma_j \neq 1$.
By the proof of Lemma \ref{lem.hwa1},
exactly one of $ac,ad,bc,bd$ is nonzero,
corresponding to the signs of the exponents in
$\gamma_i = \mu_i^{\pm 1}$ and $\gamma_j = \mu_j^{\pm 1}$.
If $\gamma_i=\mu_i$ and $\gamma_j=\mu_j$,
then $a,d \neq 0$ and so $q_{ij}=p_{ij}$.
The other cases are similar.

If $\gamma_i=1$ or $\gamma_j=1$, 
then the argument is similar. 
However, we get additional restrictions on the parameters.
For example, if $\gamma_i=1$ and $\gamma_j \neq 1$,
then we may have $ad,bd \neq 0$ implying $q_{ij}=p_{ij}=q_{ji}$,
so $q_{ij}=\pm 1$.
\end{proof}

\begin{theorem}
\label{thm.hwaiso}
Let $\gamma,\mu \in (\kk^\times)^n$ and $\bp,\bq \in \cA_n$
Then $\Phi:\hwp \rightarrow \hwm$ is an isomorphism 
if and only if $n=m$ and there exists $\ep \in \{\pm 1\}^n$ such that
$\gamma$ and $\mu$ satisfy \eqref{eq.hwa1}
while $\bp$ and $\bq$ satisfy \eqref{eq.hwa2}.
\end{theorem}

\begin{proof}
Suppose $\hwp \iso \hwm$.
By Corollary \ref{cor.hwprops}, 
$2n+1=\gk(\hwp) = \gk(\hwm)=2m+1$, so $n=m$.
By Lemma \ref{lem.hwa1}, $\gamma$ and $\mu$ satisfy \eqref{eq.hwa1}
and by Lemma \ref{lem.hwa2}, $\bp$ and $\bq$ satisfy \eqref{eq.hwa2}.

Conversely, suppose there exists $\ep \in \{\pm 1\}^n$ such that
$\lambda$ and $\mu$ satisfy \eqref{eq.hwa1}
while $\bp$ and $\bq$ satisfy \eqref{eq.hwa2}.
The isomorphisms provided in \cite{GH}
are affine and hence extend to bijective
homomorphisms $\hwp \rightarrow \hw$. 
\end{proof}

\section{Quantum matrix algebras}

Throughout, let $\{X_{ij}\}$ and $\{Y_{ij}\}$ be 
the standard generators for 
$\pma$ and $\qmamu$, respectively.
As in Section \ref{sec.deg1},
we let $I_k/I_{k-1}$ be the ideal 
in $\pma/I_{k-1}$ generated by 
the degree one normal elements.

\begin{proposition}
\label{prop.trans}
For $i>j$, let $p_{ij}=\lambda\inv q_{ji}$.
The map $\Phi:\pma \rightarrow \qma$ defined by 
$\Phi(X_{ij})=Y_{ji}$ extends to an isomorphism.
\end{proposition}

\begin{proof}
We claim the map $\Phi$ satisfies the defining relations for $\pma$,
whence $\Phi$ extends to a homomorphism.
As $\Phi$ maps onto the generators of $\qma$,
it is clearly surjective.
Injectivity now follows because $\pma$ and $\qmamu$
are domains and $\gk(\pma) = n^2 = \gk(\qmamu)$.

Case 1 ($l=i$, $m>j$)
\[
\Phi(X_{im})\Phi(X_{ij}) - p_{jm}\Phi(X_{ij})\Phi(X_{im})
	= Y_{mi}Y_{ji} - \lambda q_{mj} Y_{ji}Y_{mi} 
	= (\lambda q_{mj}q_{ii} Y_{ji}Y_{mi}) - \lambda q_{mj} Y_{ji}Y_{mi} 
	= 0.
\]

Case 2 ($l>i$, $m < j$)
\begin{align*}
\Phi(X_{lm})\Phi(X_{ij}) - \lambda p_{li} p_{jm}\Phi(X_{ij})\Phi(X_{lm})
	&= Y_{ml}Y_{ji} - \lambda (\lambda\inv q_{il})(\lambda\inv q_{mj}) Y_{ji}Y_{ml} \\
	&= Y_{ml}Y_{ji} - \lambda\inv q_{il} q_{mj} (\lambda q_{jm} q_{li} 
	Y_{ml}Y_{ji}) = 0.
\end{align*}

Case 3 ($l>i$, $m=j$)
\begin{align*}
\Phi(X_{lj})\Phi(X_{ij}) - \lambda p_{li} \Phi(X_{ij})\Phi(X_{lm})
	&= Y_{jl}Y_{ji} - \lambda (\lambda\inv q_{il}) Y_{ji}Y_{jl}
	= (q_{il} Y_{ji}Y_{jl}) - q_{il} Y_{ji}Y_{jl}
	= 0.
\end{align*}

Case 4 ($l>i$, $m>j$)
\begin{align*}
\Phi(X_{lm})\Phi(X_{ij}) &- p_{li} p_{jm}\Phi(X_{ij})\Phi(X_{lm}) 
		+ (\lambda-1)p_{li}X_{im}X_{lj}) \\
	&= Y_{ml}Y_{ji} - (\lambda\inv q_{il})(\lambda q_{mj}) Y_{ji}Y_{ml} 
		+ (\lambda-1)(\lambda\inv q_{il}) Y_{mi}Y_{jl} \\
	&= (q_{mj}q_{il} Y_{ji}Y_{ml} + (\lambda-1)q_{mj}Y_{jl}Y_{mi}) 
		- q_{il}q_{mj} Y_{ji}Y_{ml} 
		+ (\lambda-1)(\lambda\inv q_{il}) (\lambda q_{mj}q_{li}) Y_{jl}Y_{mi} = 0.
\qedhere
\end{align*}
\end{proof}

\begin{proposition}
\label{prop.flip}
Let $p_{ij}=q_{n+1-i,n+1-j}$ and $\lambda=\mu\inv$.
The map $\Phi:\pma \rightarrow \qmamu$ defined by 
$\Phi(X_{ij})=Y_{n+1-i,n+1-j}$ extends to an isomorphism.
\end{proposition}

\begin{proof}
We claim the map $\Phi$ satisfies the defining relations for $\pma$,
whence $\Phi$ extends to a bijective homomorphism by an analogous
argument as in Proposition \ref{prop.trans}.

Set $r=n+1-l$, $s=n+1-m$, $u=n+1-i$, and $v=n+1-j$.

Case 1 ($l=i, m>j$) We have $r=u$ and $s>v$. Thus,
\[ \Phi(X_{lm})\Phi(X_{ij}) - p_{jm}\Phi(X_{ij})\Phi(X_{lm}) 
		= Y_{rs}Y_{uv} - p_{jm}Y_{uv}Y_{rs}
		= (q_{vs}-p_{jm})Y_{uv}Y_{rs} = 0.
\]

Case 2 ($l>i$, $m \leq j$) We have $r<u$ and $v\leq s$. Thus,
\[ \Phi(X_{lm})\Phi(X_{ij}) - \lambda p_{li}p_{jm} \Phi(X_{ij})\Phi(X_{lm}) 
		= Y_{rs}Y_{uv} - \lambda p_{li}p_{jm} Y_{uv}Y_{rs}
		= (1-\lambda p_{li}p_{jm}\mu q_{ur}q_{sv})Y_{rs}Y_{uv} = 0.
\]

Case 3 ($l>i$, $m>j$) We have $r<u$ and $s < v$. Thus,
\begin{align*}
	\Phi(X_{lm})\Phi(X_{ij}) &- p_{li}p_{jm} \Phi(X_{ij})\Phi(X_{lm}) - (\lambda-1)p_{li}\Phi(X_{im})\Phi(X_{lj}) \\
		&= Y_{rs}Y_{uv} - p_{li}p_{jm} Y_{uv}Y_{rs} - (\lambda-1)p_{li}Y_{us}Y_{rv} \\
		&= Y_{rs}Y_{uv} - p_{li}p_{jm} (q_{ur}q_{sv}Y_{rs}Y_{uv} +(\mu-1)q_{ur}Y_{rv}Y_{us}) - (\lambda-1)p_{li}(\mu q_{ur}q_{vs})Y_{rv}Y_{us} \\
		&= -p_{li}q_{ur} ((\mu-1)p_{jm} + (\lambda-1)\mu q_{vs}) =0. \qedhere
\end{align*}
\end{proof}

The key question is whether there are any more types of isomorphisms.

\begin{lemma}
\label{lem.sum}
Suppose $\Phi:\pma \rightarrow \qmamu$ is a graded isomorphism.

(1) If $p_{jm} \neq 1$, 
then $\Phi(X_{ij})$ and $\Phi(X_{im})$
do not share any summands.

(2) If $\lambda p_{li} \neq 1$, 
then $\Phi(X_{ij})$ and $\Phi(X_{lj})$
do not share any summands.

(3) For all $i,j$, $\Phi(X_{ij})$ and $\Phi(X_{ji})$ do not share
any summands.
\end{lemma}

\begin{proof}
(1) WLOG, assume $m > j$.
Suppose $Y_{uv}$ is a summand of 
both $\Phi(X_{ij})$ and $\Phi(X_{im})$
with coefficient $a$ and $b$, respectively.
The coefficient of $Y_{uv}^2$ in 
$0=\Phi(X_{im}X_{ij}-p_{jm}X_{ij}X_{lm})$ is $ab(1-p_{jm})$.
Hence, $a=0$ or $b=0$, a contradiction.
(2) is similar.

(3) WLOG, assume $i>j$.
Suppose $Y_{uv}$ is a summand of 
both $\Phi(X_{ij})$ and $\Phi(X_{ji})$
with coefficient $a$ and $b$, respectively.
The coefficient of $Y_{uv}^2$ in
$0=\Phi(X_{ii}X_{jj}-X_{jj}X_{ii}-(\lambda-1)p_{ij}X_{ji}X_{ij})$
is $-(\lambda-1)p_{ij}ab$. 
Hence, $a=0$ or $b=0$, a contradiction.
\end{proof}

\begin{lemma}
\label{lem.lammu}
$\pma \iso \qmamu$ implies $\lambda=\mu^{\pm 1}$.
\end{lemma}

\begin{proof}
By Theorem \ref{thm.pma}, the degree one normal elements of $\pma$
(resp. $\qmamu$) are $X_{1n}$ and $X_{n1}$ (resp. $Y_{1n}$ and $Y_{n1}$).
Consequently, $\Phi(X_{1n})$ and $\Phi(X_{n1})$ 
are linear combinations of $Y_{1n},Y_{n1}$. Write
\begin{align*}
\Phi(X_{11}) &= a_1 Y_{11} + b_1 Y_{nn} + c_1 Y_{1n} + d_1 Y_{n1} + \bx,
& & \Phi(X_{1n}) = c_3 Y_{1n} + d_3 Y_{n1}, \\
\Phi(X_{nn}) &= a_2 Y_{11} + b_2 Y_{nn} + c_2 Y_{1n} + d_2 Y_{n1} + \by ,
& & \Phi(X_{n1}) = c_4 Y_{1n} + d_4 Y_{n1}.
\end{align*}
where $a_i,b_i,c_i,d_i \in \kk$
and $\bx,\by$ are linear terms not containing 
$Y_{11},Y_{nn},Y_{1n},Y_{n1}$ as summands.
By Proposition \ref{prop.flip} and Lemma \ref{lem.sum} (3),
we can reduce to the case $d_3=c_4=0$. Then
\begin{align*}
0	&= \Phi(X_{1n}X_{11} - p_{1n}X_{11}X_{1n}) \\
	&= c_3 \left[ a_1(q_{1n}-p_{1n}) Y_{11}Y_{1n}
		+ b_1(1- \mu p_{1n}q_{n1}) Y_{1n}Y_{nn} \right. \\
	&\quad + \left. c_1(1-p_{1n}) Y_{1n}^2
		+ d_1(\mu q_{1n}^2 - p_{1n}) Y_{1n}Y_{n1} \right] \\
	&\quad + \text{(additional terms not in $Y_{11}Y_{1n}, Y_{1n} Y_{nn}, 
				Y_{1n}^2, Y_{1n}Y_{n1}$)}, \\
0	&= \Phi(X_{n1}X_{11} - \lambda p_{n1}X_{11}X_{n1}) \\
	&= d_4 \left[ a_1 (\mu q_{n1}-\lambda p_{n1}) Y_{11}Y_{n1}
		+ b_1(1-\lambda p_{n1}q_{1n}) Y_{n1}Y_{nn}  \right. \\
	&\quad + \left. c_1 (1 - \lambda \mu p_{n1}q_{n1}^2) Y_{1n}Y_{n1}
		+ d_1(1-\lambda p_{n1})Y_{n1}^2  \right] \\
	&\quad + \text{(additional terms not in $Y_{11}Y_{n1}, Y_{11} Y_{1n}, 
				Y_{1n}Y_{n1}, Y_{n1}^2$)}.
\end{align*}
%\begin{align*}
%\Phi(X_{1n}X_{11} - p_{1n}X_{11}X_{1n})
%	&= a_1 \left[ c_3(q_{1n}-p_{1n}) Y_{11}Y_{1n} 
%		+ d_3 (\mu q_{n1} - p_{1n}) Y_{11}Y_{n1}\right] \\
%	&\quad + b_1 \left[ (c_3 (1- \mu p_{1n}q_{n1}) Y_{1n}Y_{nn} 
%		+ d_3 (1-p_{1n}q_{1n}) Y_{n1}Y_{nn}) 
%		\right] \\
%	&\quad + c_1 \left[ c_3 (1-p_{1n}) Y_{1n}^2
%		+ d_3 (1-\mu q_{n1}^2 p_{1n})Y_{1n}Y_{n1} \right] \\
%	&\quad + d_1 \left[ d_3 (1-p_{1n}) Y_{n1}^2
%		+ c_3 (\mu q_{1n}^2 - p_{1n}) Y_{1n}Y_{n1} \right] \\
%	&\quad + \text{(additional terms not in $Y_{11}Y_{1n}, Y_{11} Y_{n1}, 
%				Y_{1n}Y_{nn}, Y_{n1}Y_{nn}$)}. \\
%\Phi(X_{n1}X_{11} - \lambda p_{n1}X_{11}X_{n1})
%	&= a_1 \left[ (c_4 (q_{1n}-\lambda p_{n1}) Y_{11}Y_{1n} 
%		+ d_4 (\mu q_{n1}-\lambda p_{n1}) Y_{11}Y_{n1}) 	\right] \\
%	&\quad + b_1 \left[ (c_4 (1- \lambda\mu p_{n1}q_{n1}) Y_{1n}Y_{nn} 
%		+ d_4 (1-\lambda p_{n1}q_{1n}) Y_{n1}Y_{nn})  \right] \\
%	&\quad + \text{(additional terms not in $Y_{11}Y_{1n}, Y_{11} Y_{n1}, 
%				Y_{1n}Y_{nn}, Y_{n1}Y_{nn}$)}.
%\end{align*}
If $a_1\neq 0$, then $q_{1n}=p_{1n}$ and $\mu q_{n1}=\lambda p_{n1}$,
so $\mu=\lambda$. 
If $b_1 \neq 0$, then $1=\mu p_{1n}q_{n1}$ and $1=\lambda p_{n1}q_{1n}$,
so $\mu p_{1n}q_{n1} = \lambda p_{n1}q_{1n}$ implies $\lambda=\mu\inv$.
A similar argument using $X_{nn}$ in place of $X_{11}$ shows that if
$a_2\neq 0$ or $b_2\neq 0$, then $\lambda = \mu^{\pm 1}$.
Moreover, one can show that if $X_{11}$ or $X_{nn}$
are summands of either $\Phi\inv(Y_{11})$ or $\Phi\inv(Y_{nn})$
then $\lambda = \mu^{\pm 1}$.

We now reduce to the case that $a_1,a_2,b_1,b_2 = 0$. In
$0=\Phi(X_{11}X_{nn}-X_{nn}X_{11} + (\lambda-1)p_{n1} X_{1n}X_{n1})$ 
the coefficient of $Y_{1n}Y_{n1}$ is
\[ (1-\mu q_{n1}^2)(c_1d_2-c_2d_1) + (\lambda-1)p_{n1} c_3d_4. \]
If $p_{n1}, \lambda p_{n1} \neq 1$, then 
$c_1,d_1,c_2,d_2=0$ by Lemma \ref{lem.sum} (1,2)
and this coefficient reduces to $(\lambda-1)p_{n1} c_3d_4 \neq 0$,
a contradiction. Thus, $p_{n1} = 1$ or $\lambda p_{n1}=1$.
We may further assume that $q_{n1}=1$ or $\lambda q_{n1}=1$.
Thus, we have four cases to consider.
Note that
\begin{align*}
0	= \Phi(X_{n1}X_{1n} - \lambda p_{n1}^2 X_{1n}X_{n1})
	= c_3d_4 (\mu q_{n1}^2 - \lambda p_{n1}^2) Y_{1n}Y_{n1}.
\end{align*}
Thus, $\lambda p_{n1}^2=\mu q_{n1}^2$.
If $p_{n1}=q_{n1}=1$, then $\lambda=\mu$.
If $\lambda p_{n1}=q_{n1}=1$, then $p_{n1}=\mu$ so $\lambda = \mu\inv$.
The other two cases follow similarly.
\end{proof}

Suppose $\pma \iso \qmamu$. 
By Lemma \ref{lem.lammu}, $\lambda=\mu^{\pm 1}$. 
If $\lambda=\mu\inv$, then by Proposition \ref{prop.flip} 
we can replace $\qmamu$ by $\cO_{\mu\inv,\bq'}(\cM_n(\kk))$. 
Thus, it suffices to consider the isomorphism problem 
for the case where $\lambda=\mu$. 

\subsection{The $n=2$ case}
\label{sec.n2}

As in the homogenized quantized Weyl algebra case,
we consider the $n=2$ case to illustrate methods
in the general case.
We will prove an extension of \cite[Proposition 4.2]{G},
removing any restriction on the parameters.

Let $\bp,\bq \in \cA_2$ and set $p=p_{12}$ and $q=q_{12}$. We denote 
$\cO_{\lambda,\bp}(\cM_2(\kk))$ by $M_{\lambda,p}$ and 
$\cO_{\mu,\bq}(\cM_2(\kk))$ by $M_{\mu,q}$.
Recall that, by assumption, $\lambda^2,\mu^2 \neq 1$.

For simplicity, we write the basis $\{X_{ij}\}$ as
$x_1 = X_{11}$, $x_2 = X_{12}$, $x_3 = X_{21}$, and $x_4 = X_{22}$.
The defining relations for $M_{\lambda,p}$ are then
\begin{align*}
&x_2x_1 = p x_1x_2 
	& &x_4x_3 = p x_3x_4 
	& &x_3x_2 = \lambda p^2 x_2 x_3 \\
&x_3x_1 = \lambda p x_1x_3 
	& & x_4x_2 = \lambda p x_2x_4 
	& & x_4x_1 = x_1x_4 + (\lambda p - p) x_2x_3.
\end{align*}
We rewrite the basis and defining relations for $M_{\mu,q}$ similarly.

\begin{lemma}
\label{lem.mlp}
If $M_{\lambda,p} \iso M_{\lambda,q}$, 
then $q=p$ or $q=\lambda\inv p\inv$.
\end{lemma}

\begin{proof}
Let $\Phi: M_{\lambda,p} \rightarrow M_{\lambda,q}$ 
be an isomorphism.
Write $\Phi(x_i) = \sum_{k=1}^4 a_{ik} y_k$.
By Theorem \ref{thm.pma},
any normal element in $M_{\lambda,p}$ has the form 
$bx_2 + cx_3$ and similarly for $M_{\lambda,q}$. Hence,
\begin{align*}
	\Phi(x_2) = a_{22} y_2 + a_{23} y_3 \text{ and }
	\Phi(x_3) = a_{32} y_2 + a_{33} y_3,
\end{align*}
with $a_{22}a_{33}-a_{23}a_{32} \neq 0$.
By Lemma \ref{lem.sum}, $\Phi(x_2)$ and $\Phi(x_3)$
may not share any summands.
Thus, we conclude that there are two cases,
$a_{23}=a_{32}=0$, or $a_{22}=a_{33}=0$.

By Theorem \ref{thm.pma},
$\Phi$ restricts to a graded isomorphism
\[ \kk[x_1,x_4] \iso M_{\lambda,p}/\langle x_2,x_3 \rangle
\iso M_{\lambda,q}/\langle y_2,y_3 \rangle \iso \kk[y_1,y_4].\]
Hence, $a_{11}a_{44}-a_{14}a_{41} \neq 0$. Write,
$\Phi(x_1) = a_{11} y_1 + a_{14} y_4 + K$ with 
$K \in \Span_{\kk}\{y_2,y_3\}$. Then
\begin{align*}
\Phi(x_2x_1-px_1x_2)
	&= (a_{22} y_2 + a_{23} y_3)(a_{11} y_1 + a_{14} y_4 + K)
		- p(a_{11} y_1 + a_{14} y_4 + K)(a_{22} y_2 + a_{23} y_3) \\
	&= a_{11}a_{22}(q-p) y_1y_2 + a_{14}a_{22} (1-\lambda pq) y_2y_4
		+ a_{11}a_{23}(\lambda q-p) y_1y_3 + a_{14}a_{23}(1-pq) y_3y_4 \\
	&\hspace{5em}	+ (a_{22} y_2 + a_{23} y_3)K - K(a_{22} y_2 + a_{23} y_3).
\end{align*}
It is clear that the monomials $y_1y_2,y_1y_3,y_2y_4,y_3y_4$
do not appear as summands in $(a_{22} y_2 + a_{23} y_3)K - K(a_{22} y_2 + a_{23} y_3)$.
Hence, the coefficients of those monomials must be zero.

In the first case, $a_{22} \neq 0$. Then 
$a_{11} \neq 0$ implies $q=p$ and 
$a_{14} \neq 0$ implies $q=\lambda\inv p\inv$.

In the second case, $a_{23} \neq 0$. Then
$a_{11} \neq 0$ implies $q=\lambda\inv p$ and 
$a_{14} \neq 0$ implies $q=p\inv$.
%
%Since $a_{22}=0$ in this case, then we have 
%\begin{align*}
%\Phi(x_3x_2-\lambda p^2 x_2x_3)
%	= a_{23}a_{32} (1- \lambda^2 q^2 p^2) y_2y_3.
%\end{align*}
%Thus $\lambda^2 q^2 p^2 = 1$.
%Then $a_{14} \neq 0$ implies that $\lambda^2=1$, a contradiction.
%If $a_{11} \neq 0$, then repeating the above
%argument but with the commutation between $x_1$ and $x_3$
%gives $q=\lambda p$. But then $\lambda p = \lambda\inv p$, so $\lambda^2=1$,
%a contradiction.
\end{proof}

\begin{theorem}
\label{thm.mlp}
If $M_{\lambda,p} \iso M_{\mu,q}$, then
$(\mu,q)$ is one of $(\lambda,p)$, $(\lambda\inv,p\inv)$,
$(\lambda,\lambda\inv p\inv)$, or $(\lambda\inv,\lambda p)$.
\end{theorem}

\begin{proof}
If $(\mu,q)=(\lambda,p)$, then there is nothing to prove.
If $(\mu,q)=(\lambda,\lambda\inv p\inv)$, then there is an
isomorphism by Proposition \ref{prop.trans}.
If $(\mu,q)=(\lambda\inv,p\inv)$, then there is an isomorphism
by Proposition \ref{prop.flip}.
Composing these isomorphisms gives the final case.

Conversely, suppose $\Phi:M_{\lambda,p} \rightarrow M_{\mu,q}$
is a graded isomorphism.
By Lemma \ref{lem.lammu}, $\mu = \lambda^{\pm 1}$.
By Proposition \ref{prop.flip}, we may assume $\mu=\lambda$.
Lemma \ref{lem.mlp} now completes the proof.
\end{proof}

\subsection{The general case}

Assume $n>2$. We retain our original 
notation and let $\{X_{ij}\}$ and $\{Y_{ij}\}$ be 
the standard homogeneous generators for 
$\pma$ and $\qmamu$, respectively.

\begin{lemma}
\label{lem.g1}
Let $\Phi:\pma \rightarrow \qma$ be an isomorphism
that maps the ideals $(X_{1n})$ and $(X_{n1})$ to 
$(Y_{1n})$ and $(Y_{n1})$, respectively.
For all $\ell < n$,
if $X_{ij} \in I_\ell/I_{\ell-1}$,
then the ideal $(X_{ij}) + I_{\ell-1}$ is mapped
to $(Y_{ij}) + J_{\ell-1}$.
\end{lemma}

\begin{proof}
Assume throughout that $j>i$.
The case with $j<i$ is similar.
Throughout we use $\Phi$ to denote the given isomorphism
as well as any induced isomorphism
$\pma/I_k \rightarrow \qma/I_k$.

The lemma is true when $\ell=0$ by assumption.
Assume it holds for all ideals $I_k$ with $k<\ell$.
Let $S$ be the set of 
$X_{ij}$ in $I_\ell/I_{\ell-1}$,
$S^+= \{ X_{ij} \in S \mid i<j\}$, and
$S^-= \{ X_{ij} \in S \mid i>j\}$.
Define $T, T^+,T^-$ similarly for the $Y_{ij}$.

{\bf Claim 1:} 
If $X_{lm} \in S^+$ with $X_{lm} \neq X_{ij}$,
then $Y_{ij}$ is a summand of $\Phi(X_{ij})$ and
not a summand of $\Phi(X_{lm})$.

Either $l>i$ and $m>j$ or $l<i$ and $m<j$.
Assume the former.
Note that $Y_{ij}$ and $Y_{lm}$ are the only
two generators in $T^+$ such that a linear combination
of $Y_{ij}Y_{lm}$ and $Y_{lm}Y_{ij}$ is nonzero in $(Y_{im})$. 
Let $r$ be minimal such that $X_{im} \in I_r$.
Because $m > j$ then it follows that $r<\ell$.
Hence, by the inductive hypothesis,
$\Phi( (X_{im}) + I_{r-1}) = (Y_{im}) + J_{r-1}$. 
Let $U=\{Y_{uv} \in J_\ell \mid (u,v) \neq (i,j),(l,m)\}$.
Then
\begin{align*}
\Phi(X_{ij}) = a_1Y_{ij} + b_1Y_{lm} + K \text{ and }
\Phi(X_{lm}) = a_2Y_{ij} + b_2Y_{lm} + L
\end{align*}
for some $K,L \in \Span_{\kk}\{U\}$ 
and $a_i,b_j \in \kk$ with $a_1b_2-b_1a_2 \neq 0$.
We will show that $b_1=a_2=0$.

Suppose $b_1$ and $a_2$ are nonzero.
Write $\Phi(X_{im}) = \alpha Y_{im}+ J_{r-1}$, $\alpha \in \kk^\times$. Then
\begin{align}
\label{eq.pmarel}
\notag 0 &= \Phi(X_{im}X_{ij} - p_{jm}  X_{ij}X_{im} + I_{r-1}) \\
\notag   &= \alpha\left( Y_{im}(a_1Y_{ij} + b_1Y_{lm} + K) 
\notag		- p_{jm}(a_1Y_{ij} + b_1Y_{lm} + K)Y_{im} \right) + J_{r-1} \\
\tag{$\star$}   &= \alpha \left[ 
		a_1 (q_{jm} - p_{jm}) Y_{ij} Y_{im}			
		+ b_1 (1-\lambda q_{li}p_{jm}) Y_{im}Y_{lm}\right] 
		+ \alpha(Y_{im}K-p_{jm}KY_{im}) + J_{r-1}.
\end{align}
Since $Y_{ij} Y_{im}, Y_{im}Y_{lm}$ are not summands
in $(Y_{im}K-p_{jm}KY_{im})$, then our hypothesis that
$b_1 \neq 0$ implies $p_{jm}=\lambda\inv q_{il}$. 
Similarly, write $\Phi(X_{mi}) = \beta Y_{mi} + J_{r-1}$ with $\beta \in \kk^\times$. Then
\begin{align*}
0 &= \Phi(X_{lm}X_{im} - \lambda p_{li} X_{im}X_{lm}) \\
   &= \alpha\left( (a_2Y_{ij} + b_2Y_{lm} + L) Y_{im}
		- \lambda p_{li} Y_{im} (a_2Y_{ij} + b_2Y_{lm} + L) \right) + J_{r-1} \\
   &= \alpha\left( 	a_2(1-\lambda p_{li}q_{jm}) Y_{ij}Y_{im} 
		+  b_2(\lambda q_{li}-\lambda p_{li}) Y_{im}Y_{lm} \right) 
		+ \beta(LY_{im} - \lambda p_{li}Y_{im}L) + J_{r-1}.
\end{align*}
Hence, $a_2 \neq 0$ implies $p_{li} = \lambda\inv q_{mj}$. Finally,
\begin{align*}
0 &= \Phi(X_{lm}X_{ij} - p_{li}p_{jm} X_{ij}X_{lm} + (X_{im}) ) \\
   &= (a_2Y_{ij} + b_2Y_{lm} + L)(a_1Y_{ij} + b_1Y_{lm}+K) 
		- p_{li}p_{jm}(a_1Y_{ij} + b_1Y_{lm}+K)(a_2Y_{ij} + b_2Y_{lm} + L) \\
		&\hspace{5em} + (Y_{im}) + J_{r-1} \\
   &= (1-p_{li}p_{jm})(a_1a_2 Y_{ij}^2 + b_1b_2 Y_{lm}^2)
		+ \left[ a_1b_2(q_{li}q_{jm}-p_{li}p_{jm})
		+ a_2b_1(1-p_{li}p_{jm}q_{li}q_{jm}) \right] Y_{ij}Y_{lm} \\
		&\hspace{5em} + (LK-p_{li}p_{jm}KL) + (Y_{im}) + J_{r-1}.
\end{align*}
By hypothesis, $a_2b_1 \neq 0$.
Suppose further that $a_1b_2 \neq 0$.
Then $a_1a_2, b_1b_2 \neq 0$ and so $p_{li}p_{jm}=1$.
But then the coefficient of $Y_{ij}Y_{lm}$ 
reduces to $(q_{li}q_{jm}-1)(a_1b_2-a_2b_1)$,
so $q_{li}q_{lm}=1$ and
\[ 1 = p_{li}p_{jm} = (\lambda\inv q_{mj})(\lambda\inv q_{il}) = \lambda^{-2}.\]
Thus, $\lambda^2=1$, violating our hypothesis.

Hence, $a_1b_2 = 0$ Since $a_2b_1 \neq 0$, then
\[ 1 
	= p_{li}p_{jm}q_{li}q_{jm} 
	= (\lambda\inv q_{mj})(\lambda\inv q_{il})q_{li}q_{lm} = \lambda^{-2}.\]
Thus, $\lambda^2=1$, again violating our hypothesis.
 
{\bf Claim 2:} 
If $X_{lm} \in S^-$, then $Y_{lm}$ is not a summand of $\Phi(X_{ij})$. 

Suppose $Y_{lm} \in T^-$ is a summand of $\Phi(X_{ij})$.
If $(l,m)=(j,i)$, then the claim follows by Lemma \ref{lem.sum}. By Claim 1, 
$\Phi(X_{ij}) = a_1Y_{ij} + b_1Y_{lm} + K$ and 
$\Phi(X_{ji}) = a_2Y_{ji} + L$
for some $K \in \Span_{\kk}\{U\bs \{S^+ \cup \{Y_{lm}\}\} \}$,
$L \in \Span_{\kk}\{U\bs S^-\}$ and $a_i,b_j \in \kk$
with $a_1,a_2 \neq 0$.
We will show that $b_1=0$.

If $l>j$ and $m>i$, then we have
\begin{align*}
0 &= \Phi(X_{ji}X_{ij} - \lambda p_{ji}^2 X_{ij}X_{ji}) \\
  &= (a_2Y_{ji} + L)(a_1Y_{ij} + b_1Y_{lm} + K) 
  	- \lambda p_{ji}^2 (a_1Y_{ij} + b_1Y_{lm} + K)(a_2Y_{ji} + L) \\
  &= a_1a_2\lambda(q_{ji}^2 - \lambda p_{ji}^2) Y_{ij}Y_{ji}
  		+ a_2b_1\left[ (1- \lambda p_{ji}^2 q_{lj}q_{im})Y_{ji}Y_{lm}
  		- (\lambda-1)\lambda p_{ji}^2 q_{lj} Y_{jm}Y_{li}\right] \\
  &\hspace{5em}	+ a_2(Y_{ji}K-\lambda p_{ji}^2 KY_{ji})
  			+ a_1( L(a_1Y_{ij} + b_1Y_{lm})-\lambda p_{ji}^2 (a_1Y_{ij} + b_1Y_{lm})L)
  			+ (KL - \lambda p_{ji}^2 LK).
\end{align*}
Observe that $Y_{ji}, Y_{lm}$ is the unique pair 
such that $Y_{jm}Y_{li}$ is a linear combination
of $Y_{ji}Y_{lm}$ and $Y_{lm}Y_{ji}$. 
Moreover, $Y_{ji}$ and $Y_{lm}$ are not summands of $L$ or $K$.
Hence, we have $a_2b_1=0$, so $b_1=0$ as claimed.
\end{proof}

We our now ready to prove our main theorem.

\begin{theorem} 
\label{thm.mpqma}
$\pma \iso \qmamum$ if and only if
$n=m$ and one of the following cases holds:
\begin{enumerate}
\item $\lambda=\mu$ and $\bp = \bq$;
\item $\lambda=\mu$ and $p_{ij}=\lambda\inv q_{ji}$
for all $i,j$;
\item $\lambda=\mu\inv$ and $p_{ij}=q_{n+1-i,n+1-j}$ for all $i,j$;
\item $\lambda=\mu\inv$ and $p_{ij}=\lambda\inv q_{n+1-j,n+1-i}$ for all $i,j$.
\end{enumerate}
\end{theorem}

\begin{proof}
First, we will establish the indicated isomorphisms.
(1) is obtained as from the map $X_{ij} \mapsto Y_{ij}$.
(2) and (3) are obtained from Propositions \ref{prop.trans} 
and \ref{prop.flip}, respectively, while
(4) is the composition of those.

Conversely, let $\Phi:\pma \rightarrow \qmamum$ be an isomorphism.
Then $n=m$ since
\[ n^2 = \gk(\pma) = \gk(\qmamum)=m^2.\]
By Lemma \ref{lem.lammu}, we may assume $\lambda=\mu$.
By Theorem \ref{thm.pma} and Lemma \ref{lem.sum}, 
we may assume that $\Phi(X_{1n})$ and $\Phi(X_{n1})$ either fix
their respective positions or interchange them.
We consider the case where they are fixed and claim that $\bp=\bq$.
Choose $j,m \in \{1,\hdots,n\}$, $j < m$.
By Lemma \ref{lem.g1}, the coefficient of $Y_{1j}Y_{1m}$ in 
$\Phi(X_{1m}X_{1j}-p_{jm}X_{1j}X_{1m})$ is $p_{jm}-q_{jm}$
(see \eqref{eq.pmarel}).
Hence, $p_{jm}=q_{jm}$ for all $j,m$ and so $\bp=\bq$.
\end{proof}

\noindent {\bf Acknowledgements.}
The author thanks Ken Goodearl for helpful discussions
and the referee for offering several corrections and
points of clarification.

\bibliographystyle{amsplain}
%\bibliography{biblio}{}

\providecommand{\bysame}{\leavevmode\hbox to3em{\hrulefill}\thinspace}
\providecommand{\MR}{\relax\ifhmode\unskip\space\fi MR }
\providecommand{\MRhref}[2]{%
  \href{http://www.ams.org/mathscinet-getitem?mr=#1}{#2}
}
\providecommand{\href}[2]{#2}

\end{document}